\documentclass{amsart}

\usepackage{helvet, color}

\usepackage{comment,amscd,amsmath,amsxtra,amsthm,amssymb,stmaryrd,xr,mathrsfs,mathtools,enumerate,url}
\usepackage[all,cmtip]{xy}

 \DeclareFontFamily{U}{wncy}{}
    \DeclareFontShape{U}{wncy}{m}{n}{<->wncyr10}{}
    \DeclareSymbolFont{mcy}{U}{wncy}{m}{n}
    \DeclareMathSymbol{\Sh}{\mathord}{mcy}{"58}

\newtheorem{theorem}{Theorem}[section]
\newtheorem{lemma}[theorem]{Lemma}

\newtheorem{definition}[theorem]{Definition}

\numberwithin{equation}{section}

\theoremstyle{remark}
\newtheorem{remark}[theorem]{Remark}

\newcommand{\QQ}{\mathbb{Q}}

\newcommand{\Qp}{\mathbb{Q}_p}
\newcommand{\Zp}{\mathbb{Z}_p}

\newcommand{\Z}{\mathbb{Z}}
\newcommand{\p}{\mathfrak{p}}
\newcommand{\Q}{\mathbb{Q}}

\newcommand{\C}{\mathbb{C}}

\newcommand{\cO}{\mathcal{O}}

\newcommand{\Hom}{\mathrm{Hom}}
\newcommand{\Sel}{\mathrm{Sel}}

\renewcommand{\Col}{\mathrm{Col}}

\newcommand{\cX}{\mathcal{X}}

\newcommand{\ur}{\mathrm{ur}}

\newcommand{\TT}{\mathbf{T}}

\begin{document}

\title[Vanishing anticyclotomic $\mu$]{The vanishing of anticyclotomic $\mu$-invariants for non-ordinary modular forms}

\author[J.~Hatley]{Jeffrey Hatley}
\address[Hatley]{
Department of Mathematics\\
Union College\\
Bailey Hall 202\\
Schenectady, NY 12308\\
USA}
\email{hatleyj@union.edu}

\author[A.~Lei]{Antonio Lei}
\address[Lei]{Department of Mathematics and Statistics\\
150 Louis-Pasteur Pvt\\
Ottawa, ON\\
Canada K1N 6N5}
\email{antonio.lei@uottawa.ca}

\begin{abstract} Let $K$ be an imaginary quadratic field where $p$ splits. We study signed Selmer groups for non-ordinary modular forms over the anticyclotomic $\Z_p$-extension of $K$, showing that one inclusion of an Iwasawa main conjecture involving the $p$-adic $L$-function of Bertolini--Darmon--Prasanna implies that  their $\mu$-invariants vanish. This gives an alternative method to reprove a recent result of Matar on the vanishing of the $\mu$-invariants of plus and minus signed Selmer groups for elliptic curves.
\end{abstract}

\thanks{The second named author's research is supported by the NSERC Discovery Grants Program RGPIN-2020-04259 and RGPAS-2020-00096.}


\maketitle

\section{Introduction}\label{section:intro}

Let $E/\Q$ be an elliptic curve with good ordinary reduction at a prime $p$, and suppose $E[p]$ is irreducible as a $\mathrm{Gal}(\bar{\Q}/\Q)$-module. Associated to $E$ and the cyclotomic $\Z_p$-extension of $\Q$ is a $p$-primary Selmer group $\Sel(\QQ_{\mathrm{cyc}},E)$. Mazur \cite{mazur72} studied Iwasawa-theoretic properties of $E$ over $\QQ_{\mathrm{cyc}}$ by showing that $\Sel(\QQ_{\mathrm{cyc}},E)$ satisfies the so-called control theorem. This allows us to pass between arithmetic properties of $E$ over $\QQ_\mathrm{cyc}$ and those over intermediate finite subextensions of $\QQ_\mathrm{cyc}$. A famous conjecture due to Ralph Greenberg asserts that the Iwasawa $\mu$-invariant of $\Sel(\QQ_{\mathrm{cyc}},E)$ is zero \cite[Conjecture 1.11]{Gr}, meaning that its Pontryagin dual is a finitely generated $\Zp$-module.

When $E$ has good supersingular reduction at $p$, $\Sel(\QQ_{\mathrm{cyc}},E)$ is less suitable for studying the Iwasawa theory of $E$ since it does not satisfy a nice control theorem. One can instead study refined Selmer groups $\Sel^\pm(\QQ_{\mathrm{cyc}},E)$ defined by Kobayashi \cite{kobayashi03}  or $\Sel^{\#/\flat}(\QQ_{\mathrm{cyc}},E)$ defined by Sprung \cite{sprung09} depending on whether  $a_p(E)=0$ or not. As $E[p]$ is always irreducible in this case, one conjectures that the $\mu$-invariants of these refined Selmer groups vanish. Indeed, extended numerical calculations carried out by Pollack \cite{pollack03} confirm this speculation.

Some progress has been made on these conjectures; for instance, it is known that if $E_1$ and $E_2$ are elliptic curves such that $E_1[p] \simeq E_2[p]$, then the vanishing of the $\mu$-invariant of the Selmer group of one curve is equivalent to the vanishing of that of the other curve \cite{greenbergvatsal,kim09} (see also \cite{HatLei} and \cite{ponsinet} where these results have been generalized to the settings of modular forms and abelian varieties at non-ordinary primes).

While these conjectures are interesting in their own rights, they have important applications in the study of the Iwasawa main conjecture (which relates Selmer groups to analytic $p$-adic $L$-functions). This was one of the primary motivations of Greenberg and Vatsal in \cite{greenbergvatsal}. More generally, given a $p$-ordinary modular form $f$, Emerton, Pollack, and Weston showed that if the $\mu$-invariant of both the Selmer group $\Sel(\QQ_{\mathrm{cyc}},f)$ and the analytic $p$-adic $L$-function $L_p(f)$ vanish, then the main conjecture holds for every modular form in the Hida family to which $f$ belongs \cite[Corollary 1]{epw}.

While one may confirm the vanishing of $\mu(\Sel(\QQ_{\mathrm{cyc}},E))$ for particular elliptic curves, general vanishing results remain elusive. If we instead fix an imaginary quadratic field $K/\Q$ which satisfies a Heegner hypothesis and consider the corresponding \textit{anticyclotomic} $\Z_p$-extension, much more is known.\footnote{In the absence of the Heegner hypothesis, vanishing $\mu$ results were obtained by Pollack and Weston \cite{pollackweston11}.} For instance, in the $p$-ordinary case, the authors proved a general vanishing result for $\mu$-invariants associated to modular forms \cite[Theorem 4.5]{MRL2}. In the supersingular setting, a recent result of Matar gives a vanishing result for anticyclotomic $\mu$-invariants associated to an elliptic curve $E$ with $a_p(E)=0$, see \cite[Theorem B]{Matar2021}.

The goal of this paper is to offer a method for extending these results to non-ordinary modular forms, conditional on one inclusion of an anticyclotomic Iwasawa main conjecture (see hypothesis \textbf{(H-$\subseteq$)} in Section \ref{sec:lemmas}). We work with the signed $\#/\flat$-Selmer groups defined in \cite{BL2}, whose construction generalizes that of the $\pm$-Selmer groups in the elliptic curve case. See Theorem~\ref{thm:mu-vanishes} for our main result. Our approach is similar to the one employed in \cite{MRL2} and gives a new  proof of \cite[Theorem~B]{Matar2021} under a different set of hypothesis (see Remark~\ref{rk:new-Matar}). In addition to the Iwasawa main conjecture, we rely on the study of auxiliary Selmer groups and  a result on the mod $p$ non-vanishing  of generalized Heegner cycles by Hsieh \cite{Hsieh-Doc}.

On combining our results from \cite{HL-MRL, MRL2}, we expect that Theorem~\ref{thm:mu-vanishes} will have applications in anticyclotomic Iwasawa main conjectures in the non-ordinary case. In order to obtain results in this direction, one will have to prove that the signed Selmer groups do not admit non-trivial submodules of finite index, generalizing our previous joint work with Vigni \cite{HLV}. We plan to study these questions in the future.

\subsection*{Acknowledgements}
We thank the anonymous referee and Luochen Zhao for valuable comments and suggestions on an earlier version of this article.

\section{Hypotheses and notation}\label{sec:notation}

\subsection{Running hypotheses}
Throughout this paper, $p \geq 5$ denotes a fixed odd prime. We will consider a normalized eigen-newform of level $\Gamma_0(N)$ and weight $2r$, where $N>3$ is squarefree, $p\nmid N$, and $r \geq 1$ is an integer. We  assume that $p>2r$; since $f$ is assumed non-ordinary at $p$, it follows from a result of Fontaine and Edixhoven that any twist of the associated residual representation is absolutely irreducible. Let $K$ be an imaginary quadratic extension of $\Q$ with discriminant $d_K$ coprime to $Np$ in which $p=\p \bar{\p}$ and every prime dividing $N$ splits. We assume $K$ has unit group $\{ \pm1\}$ and that $d_k$ is odd or divisible by $8$.

If all of these conditions are met, we say that $(f,K,p)$ is \textit{admissible}.

\subsection{Running notation}
Fix an embedding $K \hookrightarrow \C$ as well as an embedding $\overline{\Q} \hookrightarrow \C_p$
 such that $\p$ lands inside the maximal ideal of $\cO_{\C_p}$, the ring of integers of $\C_p$. 

Let $\mathfrak{F}=\Q_p(\{a_n(f)\})$ denote the finite extension of $\Q_p$ (inside of $\mathbb{C}_p$) generated by the Fourier coefficients of $f$, and let $\cO$ denote its valuation ring with a fixed uniformizer $\varpi$. We will assume that $f$ is non-ordinary at $p$, in the sense that $a_p(f)$ is not a unit in $\cO$.

Let $W_f$ be a realization of Deligne's $2$-dimensional $\mathfrak{F}$-representation associated to $f$. If we normalize such that the cyclotomic character has Hodge-Tate weight $+1$, then $W_f$ has Hodge-Tate weights $\{0,1-k\}$. Fix a Galois-stable $\cO$-lattice $R_f \subset W_f$. Write $W_f^\ast = \Hom(W_f,\mathfrak{F})$ and $R_f^\ast = \Hom(R_f,\cO)$ for the dual representations. Denote by $T_f=R_f^\ast(1-r)$ the central critical twist of Deligne's representation, and by $A_f$ its Pontryagin dual $T_f^\vee = \mathrm{Hom}(T_f,\Q_p/\Z_p)$. Our normalization means that $V_f$ is self-dual and we have $A_f=V_f/T_f$ as $G_\QQ$-modules.

 Denote by $K_\infty$ the anticyclotomic $\Z_p$-extension of $K$ with Galois group $\Gamma$. Let $\Lambda$ denote the Iwasawa algebra $\cO[[\Gamma]]$. We write $\mathbf{T}=T_f \otimes \Lambda^\iota$, where $\iota$ denotes the involution on $\Lambda$ sending a group-like element to its inverse.


\section{Signed anticyclotomic Selmer groups}\label{sec:signed-selmer}

In this section we briefly recall the construction and properties of some signed anticyclotomic Selmer groups associated to our modular form $f$. For proofs and more details, the reader is referred to \cite[\S2 and \S4]{BL2}.

\begin{remark} In \cite{BL2}, the authors work with the Galois representation attached to $f$ twisted by a $p$-distinguished (in particular nontrivial) ray class character over $K$. However, this condition is only used in the study of an Euler system in \S3 of loc. cit., which we will not need for our purposes. In the present paper, we will set the ray class character to be the trivial character.   Note that the signed Selmer groups defined in  loc. cit. with trivial ray class character and $r=1$ have also been studied in \cite{CCSS}.
\end{remark}

In \cite{BL2}, Sections $2.3$ and $2.5$ are devoted to defining $\#/\flat$-Coleman maps associated to $T_f$, which are used to define signed Selmer groups in the non-ordinary setting. In particular, for each $\mathfrak{q} \in \{\p, \bar{\p}\}$, Shapiro's lemma gives an identification
\begin{equation}\label{isom:shapiro}
H^1(K_\mathfrak{q},\mathbf{T})\simeq \varprojlim_n \bigoplus_{v \mid \mathfrak{q}}  H^1(K_{n,v},T_f).
\end{equation} By decomposing Perrin-Riou's big logarithm map via the theory of Wach modules, for each $\mathfrak{q} \in \{\p, \bar{\p}\}$ and $\bullet \in \{ \#,\flat \}$ we obtain maps
\[
\mathrm{Col}_{\bullet,\mathfrak{q}} \colon H^1(K_\mathfrak{q},\mathbf{T}) \rightarrow \Lambda.
\]
In fact, two-variable Coleman maps over the Iwasawa algebra of the $\Zp^2$-extension of $K$ have been constructed in \cite{BL2}. For our purposes, we are taking the anticyclotomic specializations of those maps.

\begin{remark}
 In general, one may define a pair of Coleman maps with respect to any Wach module basis satisfying \cite[Theorem~3.5]{leiloefflerzerbes10}. The Coleman maps defined in \cite{BL} are built out of the Wach module basis constructed in \cite[Proposition~V.2.3]{berger04}, which only depends on $p$, $k$ and $a_p(f)$.
 \end{remark}

The relevant local Selmer structures are defined as follows.
\begin{definition}
For each $\mathfrak{q} \in \{ \p, \bar{\p}\}$ and $\mathcal{L}_\mathfrak{q} \in \{0,\emptyset,\#,\flat\}$, define the local structure
\[
H^1_{\mathcal{L}_\mathfrak{q}}(K_\mathfrak{q},\mathbf{T}) = \begin{cases}
H^1(K_\mathfrak{q},\mathbf{T}) & \text{if}\ \mathcal{L}_\mathfrak{q}=\emptyset, \\
\{0\} & \text{if}\ \mathcal{L}_\mathfrak{q}=0 ,\\
\ker(\mathrm{Col}_{\bullet,\mathfrak{q}}) & \text{if}\ \bullet \in \{ \#, \flat \} .
\end{cases}.
\]
\end{definition}
For any place $v \nmid p$ of $K_\infty$, we write
\[
{H^1_{\mathcal{L}_v}(K_v,\mathbf{T})}  = \mathrm{ker}\left( H^1(K_v,\mathbf{T}) \rightarrow H^1(K_v^\mathrm{ur},\mathbf{T}) \right)
\]
for the usual unramified condition.

In what follows, $\Sigma$ is any finite set of places of $K$ containing those dividing $Np\infty$, and $K_\Sigma$ is the maximal extension of $K$ unramified outside of $\Sigma$.

\begin{definition}
For $\mathcal{L}=( \mathcal{L}_\p, \mathcal{L}_{\bar{\p}} )\in \{0,\emptyset,\#,\flat\}^2$, we define the (compact) Selmer groups
\[
\mathrm{Sel}_\mathcal{L}(K,\mathbf{T}) = \ker \left( H^1(K_\Sigma /K, \mathbf{T}) \rightarrow \bigoplus_{\lambda \in \Sigma} \frac{H^1(K_\lambda,\mathbf{T})}{H^1_{\mathcal{L}_\lambda}(K_\lambda,\mathbf{T})} \right).
\]
\end{definition}

For each place $\lambda$, let us write $$H^1(K_{\infty,\lambda},A_f)= \bigoplus_{w|\lambda} H^1(K_{\infty,w},A_f).$$ When $\lambda \nmid p$ or $\mathcal{L}_\lambda \notin \{ 0, \emptyset\}$, we obtain a dual Selmer structure for $A_f$ at $\lambda$ by taking $$H^1_{\mathcal{L}_\lambda}(K_{\infty,\lambda},A_f)\subset  H^1(K_{\infty,\lambda},A_f)$$ to be the orthogonal complement $H^1_{\mathcal{L}_\lambda}(K_\lambda,\mathbf{T})^\perp$ of $H^1_{\mathcal{L}_\lambda}(K_\lambda,\mathbf{T})$ with respect to local Tate duality at $\lambda$
\begin{equation}\label{eq:tate-local-duality}
H^1(K_\lambda,\mathbf{T})\times  H^1(K_{\infty,\lambda},A_f)\rightarrow \mathfrak{F}/\cO.
\end{equation}
When $\lambda \mid p$ and $\mathcal{L}_\lambda \in \{ 0, \emptyset\}$, we set
\[
H^1_{\mathcal{L}_\mathfrak{\lambda}}(K_\mathfrak{\lambda},A_f) = \begin{cases}
H^1(K_\mathfrak{\lambda},A_f) & \text{if}\ \mathcal{L}_\mathfrak{\lambda}=\emptyset \\
\{0\} & \text{if}\ \mathcal{L}_\mathfrak{\lambda}=0
\end{cases}.
\]

\begin{definition}\label{def:dual-selmer} For $\mathcal{L}=( \mathcal{L}_\p, \mathcal{L}_{\bar{\p}} )\in \{0,\emptyset,\#,\flat\}^2$, we define the Selmer groups of $A_f$ over $K_\infty$ by
\[
\mathrm{Sel}_\mathcal{L}(K_\infty,A_f) = \ker \left( H^1(K_\Sigma /K_\infty, A_f) \rightarrow \bigoplus_{\lambda \in \Sigma} \frac{H^1(K_{\infty,\lambda},A_f)}{H^1_{\mathcal{L}_\lambda}(K_{\infty,\lambda},A_f)} \right).
\]
\end{definition}


We will be interested in the Pontryagin duals of these Selmer groups, so we set the following notation.
\begin{definition}
For any set of local conditions $\mathcal{L}$, we write $\mathcal{X}_{\mathcal{L}}(f) := \mathrm{Sel}_\mathcal{L}(K_\infty,A_f)^\vee$ for the Pontryagin dual of the corresponding Selmer group.
\end{definition}

\begin{remark}\label{rmk:relation-to-plus-minus}To motivate the study of these Selmer groups, we mention that the maps $\mathrm{Col}_{\#/\flat,\mathfrak{q}}$ defined using Wach modules are generalizations of the Coleman maps studied by Kobayashi \cite{kobayashi03} and Sprung \cite{sprung09}, which were constructed using Honda theory of formal groups. In particular, when $f$ corresponds to an elliptic curve defined over $\Q$ and $a_p(f)=0$, we recover the $\pm$-Selmer groups that have generated much study in the literature. See \cite[\S5.2]{leiloefflerzerbes10} where the comparison between the two approaches (over the cyclotomic extension of $\QQ$) is discussed. Since the two-variable Coleman maps in \cite{BL} are defined via inverse limits of the one-variable cyclotomic Coleman maps over unramified extensions over $\Qp$, it can be checked that the plus and minus conditions over the $\Zp^2$-extensions are given by the   ``jumping conditions" of Kobayashi and Kim in \cite{kobayashi03,kim13}. In particular, the resulting conditions over the anticyclotomic $\Zp$-extension of $K$ also agree.
\end{remark}

\begin{remark}
In the proof of Lemma \ref{lem:castella} below, we will need to consider the Selmer groups $\mathrm{Sel}_\mathcal{L}(K_\infty,A_f(\psi))$, where $A_f(\psi)$ is the twist of $A_f$ by a finite order character $\psi$ of $\Gamma$. These Selmer groups are defined by modifying Definition \ref{def:dual-selmer} in the obvious way.
\end{remark}

\section{Some preliminary lemmas}\label{sec:lemmas}

We will now collect several preliminary lemmas which will be necessary for the proof of Theorem \ref{thm:mu-vanishes} in the next section. 
The first lemma we prove is a generalization of \cite[Lemma 2.3]{Castella}.

\begin{lemma}\label{lem:castella} Assume $(f,K,p)$ is admissible. Let $\star\in\{\#,\flat\}$. The following statements are true.
\begin{enumerate}
    \item $\mathrm{rank}_\Lambda \mathcal{X}_{(\star,0)}(f)=0$.
    \item $\mathrm{rank}_\Lambda \mathcal{X}_{(\star,\emptyset)}(f)=1$
    \item We have an equality of $\Lambda$-ideals
\[
\mathrm{char}_\Lambda \left( \mathcal{X}_{( \star,0)}(f)\right) = \mathrm{char}_\Lambda \left( \mathcal{X}_{(\star,\emptyset)}(f)\right)_{\mathrm{tor}},
\]
where $\left(\mathcal{X}_{(\star,\emptyset)}(f)\right)_{\mathrm{tor}}$ denotes the maximal $\Lambda$-torsion submodule of $\mathcal{X}_{(\star,\emptyset)}(f)$.
\end{enumerate}
\end{lemma}

\begin{proof}
Recall \cite[Theorem~1.5]{KobOta} that $\mathcal{X}_{(\emptyset,0)}(f)$ is a finitely generated torsion $\Lambda$-module. Since $\mathcal{X}_{(\star,0)}(f)$ is a quotient of $\mathcal{X}_{(\emptyset,0)}(f)$ (after taking the dual of the inclusion $\Sel_{(\star,0)}(K_\infty,A_f)\hookrightarrow\Sel_{(\emptyset,0)}(K_\infty,A_f) $), statement $(1)$ follows. Using this fact, statements $(2)$ and $(3)$ can be proven in the same way as \cite[Lemma 2.3]{Castella} and \cite[Theorem 1.2.2]{AgboolaHoward}. There are only a few places where the argument uses details specific to the particular local conditions under consideration, so we briefly explain those points.

Let $\psi \colon \Gamma \rightarrow \cO_L^\times$ be a finite order character, where $\cO_L$ is the ring of integers inside some finite extension $L/\Q_p$, and extend it linearly to a map $\Lambda \colon \Gamma \rightarrow \cO_L$. Let $A_f(\psi)=A_f \otimes_{\Lambda,\psi} \cO_L$ be equipped with the diagonal $G_K$-action.

Note that the $\emptyset$ and $0$ local conditions which we defined for $\mathbf{T}$ are dual to each other under \eqref{eq:tate-local-duality}. Since $T_f / \varpi^i T_f \simeq A_f[\varpi^i]$ for all $i \geq 0$, we may apply \cite[Lemma 3.5.3 and
Theorem 4.1.13]{MazurRubin} to obtain an isomorphism
\[
H^1_{(\star,\emptyset)}(K_\infty, A_f(\psi))[\varpi^i] \simeq (L/\cO_L)^r [\varpi^i] \oplus H^1_{(\star,0)}(K_\infty, A_f(\psi^{-1}))[\varpi^i].
\]
Here we have written $H^1_{\mathcal{L}}(K_\infty, A_f(\psi^j))$ for the subgroup of $\Sel_{\mathcal{L}}(K_\infty, A_f(\psi^j))$ consisting of cohomology classes whose restriction to $\lambda \mid p$ lies in $H^1_{\mathcal{L}_{\lambda}}(K_\lambda,A_f(\psi^j))_{\mathrm{div}}$, and $r$ is the core rank of the $\{\star,\emptyset\}$ local conditions \cite[Definition 4.1.11]{MazurRubin}. In particular, by Wiles' formula \cite[Proposition 2.3.5]{MazurRubin}, $r$ is given by
\[
\mathrm{corank}_{\cO} \left(\ker(\mathrm{Col}_{\star,\mathfrak{p}})^\perp \right)+ \mathrm{corank}_{\cO}H^1(K_{\bar{\mathfrak{p}}},A_f(\psi)) - \mathrm{corank}_{\cO}H^0(K_v,A_f(\psi)),
\]
where $v$ denotes the unique archimedean place of $K$.
The first term is equal to $1$; this follows from \cite[Lemma 2.16]{BL2}, which tells us that the image of $\mathrm{Col}_{\star,\mathfrak{p}}$ is of rank 1 over $\Lambda$, so the kernel is of rank 1 as well. The second term is equal to $2$ by the local Euler characteristic formula, and the third term is visibly $2$. Thus $r=1$, and letting $i \rightarrow \infty$ gives an isomorphism
\begin{equation}\label{eq:ah-lemma}
H^1_{\{\star,\emptyset\}}(K_\infty, A_f(\psi)) \simeq (L/\cO_L) \oplus H^1_{\{\star,0\}}(K_\infty, A_f(\psi^{-1})).
\end{equation}
With \eqref{eq:ah-lemma} established, the rest of the arguments are formal and do not rely on any particular details about the local conditions defining the Selmer groups. Thus, the proof concludes in the same way as \cite[Lemma 2.3]{Castella}.
\end{proof}

We now introduce a technical assumption, the validity of which is discussed below in Remark \ref{rmk:hypoth}.

\begin{itemize}
\item[\textbf{(H-$\star$)}] $\mathrm{rank}_\Lambda \mathcal{X}_{(\star,\star)}(f)$=1 for some $\star\in\{\#,\flat\}$.
\end{itemize}

\begin{lemma}\label{lem:short-exact-corank-1}
Suppose $(f,K,p)$ is admissible and that \textbf{(H-$\star$)} holds. Then there exists a short exact sequence
\[
0 \rightarrow D \rightarrow \mathcal{X}_{(\star,\emptyset)}(f) \rightarrow \mathcal{X}_{(\star, \star)}(f) \rightarrow 0,
\]
where $D$ is a torsion $\Lambda$-module.
\end{lemma}

\begin{proof}
Upon dualizing the natural inclusion map $\Sel_{(\star,\star)}(K_\infty, A_f) \hookrightarrow \Sel_{(\star,\emptyset)}(K_\infty, A_f)$, we obtain a surjection $\mathcal{X}_{(\star,\emptyset)}(f) \rightarrow \mathcal{X}_{(\star, \star)}(f)$. By assumption, $\mathrm{rank}_\Lambda \mathcal{X}_{(\star,\star)}=1$. By Lemma \ref{lem:castella}(2) we also have  $\mathrm{rank}_\Lambda \mathcal{X}_{(\star,\emptyset)}(f)$=1. The lemma is now immediate from rank considerations.
\end{proof}

\begin{remark}\label{rmk:hypoth}
When $r=1$, \textbf{(H-$\star$)} is shown to hold in \cite[Theorem 5.7]{CCSS}. See also \cite[Theorem~1.4]{LongoVigni2} in the case of elliptic curves with $a_p=0$. The proofs of these results require the machinery of Kolyvagin systems and it goes back to Howard \cite{howard,howard2}. When $r>1$ and $a_p(f)=0$, it is expected that \textbf{(H-$\star$)} holds for both $\#$ and $\flat$ (see \cite[Remark 4.11 and Theorem~4.14(v)]{BL}).
\end{remark}

For the rest of the paper, let $\mathfrak{L}(f/K)\in \Lambda^{\mathrm{ur}}$ denote the $p$-adic $L$-function attached to $f$ over $K$ defined by Brako\v{c}evi\'{c} and Bertolini--Darmon--Prasanna in \cite{miljan,BDP}, where $\Lambda^{\mathrm{ur}}$ is the Iwasawa algebra $\Lambda\otimes_{\Zp}\Zp^\ur$, with $\Zp^\ur$ being the $p$-adic completion of the ring of integers of the maximal unramified extension of $\Qp$. We note that this $p$-adic $L$-function is defined using two periods (a complex period $\Omega_\infty$ and  a $p$-adic period $\Omega_p$) which are independent of $f$ (see \cite[\S5.2]{miljan} and \cite[\S2.2]{KobOta}). Indeed, $\Omega_\infty$ depends only on $K$ and $\Omega_p$ is a $p$-adic unit. In particular, the $\mu$-invariant of $\mathfrak{L}(f/K)$ is well-defined.

We  assume the following divisibility from the so-called BDP anticyclotomic Iwasawa main conjecture.

\begin{itemize}
\item[\textbf{(H-$\subseteq$)}] \quad $ \mathfrak{L}(f/K)^2 \Lambda^{\mathrm{ur}}\subseteq \mathrm{char}_\Lambda\left( \mathcal{X}_{(\emptyset,0)}(f) \right) \otimes_\Lambda \Lambda^{\mathrm{ur}}.$
\end{itemize}

\begin{remark}
In fact, the BDP anticyclotomic Iwasawa main conjecture predicts that there is an equality $$\mathfrak{L}(f/K)^2 \Lambda^{\mathrm{ur}}= \mathrm{char}_\Lambda\left( \mathcal{X}_{(\emptyset,0)}(f) \right) \otimes_\Lambda \Lambda^{\mathrm{ur}}.$$
 When $r=1$, Castella--Çiperiani--Skinnner--Sprung gave a proof of this conjecture under certain hypotheses in \cite[Theorem~5.8]{CCSS}.
 For a general $r$, Kobayashi--Ota showed in \cite[Theorem 1.5]{KobOta} that
\[
\mathfrak{L}(f/K)^2 \in \mathrm{char}_\Lambda\left( \mathcal{X}_{(\emptyset,0)}(f)\right) \otimes_\Lambda \Lambda^{\mathrm{ur}} \otimes_{\Zp} \Qp .
\]
 In a forthcoming joint work with Luochen Zhao, we study the integrality of this inclusion for general $r$. We thank the referee for pointing out that the tensor product with $\Qp$ creates problems for studying the $\mu$-invariant.
\end{remark}

\begin{remark}
In \cite{MRL2}, the authors mistakenly cited the weaker version of \textbf{(H-$\subseteq$)} given in \cite[Theorem 1.5]{KobOta} while proving a vanishing $\mu$ result for modular forms with good \textit{ordinary} reduction. In the ordinary case, it is possible to appeal to other sources which do handle the integrality issue, so the validity of \cite[Theorem~4.5]{MRL2} is unaffected. We explain this alternate argument here.

In the $p$-ordinary elliptic curve case, the divisibility of \textbf{(H-$\subseteq$)} has been established by Castella; see displayed equation (A.9) in \cite[Proof of Theorem 3.4]{Castella}. In this argument, Castella refers to Theorem A.5 of op. cit., which is one inclusion of the Heegner point main conjecutre formulated by Perrin-Riou in \cite{Perrin-Riou} and it has been proved by \cite[Theorem B]{howard2}. Replacing Howard's result with \cite[Theorem 3.5]{LongoVigni}, the argument of Castella carries over to prove \textbf{(H-$\subseteq$)}  for general modular forms.

\end{remark}

\section{Main result}\label{sec:main-iwasawa-invariants}

Recall that for $M$ a finitely-generated $\Lambda$-module, there exists a pseudo-isomorphism (that is, a $\Lambda$-morphism with finite kernel and cokernel)
\begin{equation}\label{eq:pseudo}
M\sim \Lambda^{\oplus r}\oplus \bigoplus_{i=1}^s \Lambda/(\varpi^{a_i})\oplus \bigoplus_{j=1}^t\Lambda\big/\bigl(F_j^{n_j}\bigr)
\end{equation}
for suitable integers $r,s,t\ge 0$, $a_i,n_j\ge1$ and irreducible Weierstrass polynomials $F_j\in\cO[X]$. We then define the \emph{$\mu$-invariant} and the \emph{$\lambda$-invariant} of $M$ by
\[ \mu(M):=\sum_{i=1}^s a_i,\qquad \lambda(M):=\sum_{j=1}^tn_j\deg(F_j). \]
Our focus of study in this paper is the $\mu$-invariant. For any set of local conditions $\mathcal{L}$ as in Section \ref{sec:signed-selmer}, we write
\[
\mu_\mathcal{L}(f)=\mu\left(\mathcal{X}_\mathcal{L}(f)\right).
\]

The following is our main result.

\begin{theorem}\label{thm:mu-vanishes}
Suppose $(f,K,p)$ is admissible and that both \textbf{(H-$\star$)} and \textbf{(H-$\subseteq$)} hold. 
If 
$$\mathcal{L} \in \{(\emptyset,0), (\star,0), (\star,\emptyset), (\star,\star) \},$$
then $\mu_{\mathcal{L}}(f)=0$.
\end{theorem}
\begin{proof} 
We know from \cite[Theorem B]{Hsieh-Doc} that $\mathfrak{L}(f/K)\in \Lambda^\ur\setminus\varpi\Lambda^\ur$, hence $\mu_{(\emptyset,0)}(f)=0$ thanks to  \textbf{(H-$\subseteq$)}.

Since $\mathcal{X}_{(\star,0)}(f)$ is a quotient of $\mathcal{X}_{(\emptyset,0)}(f)$, there is a short exact sequence
\[
0 \rightarrow C \rightarrow  \mathcal{X}_{(\emptyset,0)}(f) \twoheadrightarrow \mathcal{X}_{(\star,0)}(f) \rightarrow 0.
\]
As $\mathcal{X}_{(\emptyset,0)}(f)$ is a finitely generated torsion $\Lambda$-module by \cite[Theorem~1.5]{KobOta}, we see that $C$ must also be torsion over $\Lambda$. By \cite[Proposition 2.1]{HL-MRL}, $\mu$-invariants are additive in short exact sequences of $\Lambda$-modules where the first term is torsion. Since $\mu_{(\emptyset,0)}(f)=0$ and $\mu$-invariants are nonnegative, we deduce that $\mu_{(\star,0)}(f)=0$. This in turn allows us to deduce that $\mu_{(\star,\emptyset)}(f)=0$ thanks to Lemma \ref{lem:castella}(3).

Recall from Lemma \ref{lem:short-exact-corank-1} that we have an exact sequence
\[
0 \rightarrow D \rightarrow \mathcal{X}_{(\star,\emptyset)}(f) \rightarrow \mathcal{X}_{(\star, \star)}(f) \rightarrow 0,
\]
where $D$ is a torsion $\Lambda$-module. Applying \cite[Proposition 2.1]{HL-MRL} once again yields $\mu_{(\star,\star)}(f)=0$.
\end{proof}

\begin{remark}
\label{rk:new-Matar}
As mentioned in Remark \ref{rmk:relation-to-plus-minus}, when $f$ is associated to an elliptic curve over $\Q$ with $a_p(f)=0$, the $\#/\flat$-Selmer groups coincide with the $\pm$-Selmer groups. Thus, when the hypotheses of \cite[Theorem~5.8]{CCSS} hold,  Theorem \ref{thm:mu-vanishes} recovers Matar's result \cite[Theorem B]{Matar2021} since \cite[Theorem~1.4]{LongoVigni2} tells us that \textbf{(H-$\star$)} holds for both choices of sign.
\end{remark}

\bibliographystyle{amsalpha}
\bibliography{references}

\providecommand{\bysame}{\leavevmode\hbox to3em{\hrulefill}\thinspace}
\providecommand{\MR}{\relax\ifhmode\unskip\space\fi MR }
\providecommand{\MRhref}[2]{%
  \href{http://www.ams.org/mathscinet-getitem?mr=#1}{#2}
}
\providecommand{\href}[2]{#2}
\begin{thebibliography}{C{\c{C}}SS18}

\bibitem[AH06]{AgboolaHoward}
Adebisi Agboola and Benjamin Howard, \emph{Anticyclotomic {I}wasawa theory of
  {CM} elliptic curves}, Ann. Inst. Fourier (Grenoble) \textbf{56} (2006),
  no.~4, 1001--1048.

\bibitem[BDP13]{BDP}
Massimo Bertolini, Henri Darmon, and Kartik Prasanna, \emph{Generalized
  {H}eegner cycles and $p$-adic {R}ankin $l$-series}, Duke Math J. \textbf{162}
  (2013), no.~6, 1033--1148.

\bibitem[Ber04]{berger04}
Laurent Berger, \emph{Limites de repr\'esentations cristallines}, Compos. Math.
  \textbf{140} (2004), no.~6, 1473--1498.

\bibitem[BL18]{BL}
K\^{a}z{\i}m B\"uy\"ukboduk and Antonio Lei, \emph{Anticyclotomic
  {$p$}-ordinary {I}wasawa theory of elliptic modular forms}, Forum Math.
  \textbf{30} (2018), no.~4, 887--913.

\bibitem[BL21]{BL2}
\bysame, \emph{Iwasawa theory of elliptic modular forms over imaginary
  quadratic fields at non-ordinary primes}, Int. Math. Res. Not. IMRN (2021),
  no.~14, 10654--10730.

\bibitem[Bra11]{miljan}
Miljan Brako\v{c}evi\'{c}, \emph{Anticyclotomic {$p$}-adic {$L$}-function of
  central critical {R}ankin-{S}elberg {$L$}-value}, Int. Math. Res. Not. IMRN
  (2011), no.~21, 4967--5018.

\bibitem[Cas17]{Castella}
Francesc Castella, \emph{$p$-adic heights of {H}eegner points and
  {B}eilinson-{F}lach elements}, J. London Math. Soc. \textbf{6} (2017), no.~1,
  1--23.

\bibitem[C{\c{C}}SS18]{CCSS}
Francesc Castella, Mirela {\c{C}}iperiani, Christopher Skinner, and Florian
  Sprung, \emph{On the {I}wasawa main conjectures for modular forms at
  non-ordinary primes}, 2018, preprint, arXiv:1804.10993.

\bibitem[EPW06]{epw}
Matthew Emerton, Robert Pollack, and Tom Weston, \emph{Variation of {I}wasawa
  invariants in {H}ida families}, Invent. Math. (2006), no.~163, 523--580.

\bibitem[Gre99]{Gr}
Ralph Greenberg, \emph{Iwasawa theory for elliptic curves}, Arithmetic theory
  of elliptic curves (Cetraro, 1997), Lecture Notes in Math., vol. 1716,
  Springer, Berlin, 1999, pp.~51--144.

\bibitem[GV00]{greenbergvatsal}
Ralph Greenberg and Vinayak Vatsal, \emph{On the {I}wasawa invariants of
  elliptic curves}, Invent. Math. \textbf{142} (2000), no.~1, 17--63.

\bibitem[HL19a]{HatLei}
Jeffrey Hatley and Antonio Lei, \emph{Arithmetic properties of signed {S}elmer
  groups at non-ordinary primes}, Ann. Inst. Fourier (Grenoble) \textbf{69}
  (2019), no.~3, 1259--1294.

\bibitem[HL19b]{HL-MRL}
\bysame, \emph{Comparing anticyclotomic {S}elmer groups of positive coranks for
  congruent modular forms}, Math. Res. Lett. \textbf{26} (2019), no.~4,
  1115--1144.

\bibitem[HL21]{MRL2}
\bysame, \emph{Comparing anticyclotomic {S}elmer groups of positive coranks for
  congruent modular forms - {P}art {II}}, J. Number Theory \textbf{229} (2021),
  342--363.

\bibitem[HLV22]{HLV}
Jeffrey Hatley, Antonio Lei, and Stefano Vigni, \emph{{$\Lambda$}-submodules of
  finite index of anticyclotomic plus and minus {S}elmer groups of elliptic
  curves}, Manuscripta Math. \textbf{167} (2022), no.~3-4, 589--612.

\bibitem[How04a]{howard}
Benjamin Howard, \emph{The {H}eegner point {K}olyvagin system}, Compos. Math.
  \textbf{140} (2004), no.~6, 1439--1472.

\bibitem[How04b]{howard2}
\bysame, \emph{Iwasawa theory of {H}eegner points on abelian varieties of
  $\mathrm{GL}_2$-type}, Duke Math J. \textbf{124} (2004), no.~1, 1--45.

\bibitem[Hsi14]{Hsieh-Doc}
Ming-Lun Hsieh, \emph{Special values of anticyclotomic {R}ankin-{S}elberg
  {$L$}-functions}, Doc. Math. \textbf{19} (2014), 709--767.

\bibitem[Kim09]{kim09}
Byoung~Du Kim, \emph{The {I}wasawa invariants of the plus/minus {S}elmer
  groups}, Asian J. Math. \textbf{13} (2009), no.~2, 181--190.

\bibitem[Kim13]{kim13}
\bysame, \emph{The plus/minus {S}elmer groups for supersingular primes}, J.
  Aust. Math. Soc. \textbf{95} (2013), no.~2, 189--200.

\bibitem[KO20]{KobOta}
Shinichi Kobayashi and Kazuto Ota, \emph{Anticyclotomic main conjecture for
  modular forms and integral {P}errin-{R}iou twists}, In Development of Iwasawa
  Theory – the Centennial of K. Iwasawa's Birth (M. Kurihara et al, eds),
  Adv. Stud. Pure Math., Math. Soc. Japan (2020), 537--594.

\bibitem[Kob03]{kobayashi03}
Shinichi Kobayashi, \emph{Iwasawa theory for elliptic curves at supersingular
  primes}, Invent. Math. \textbf{152} (2003), no.~1, 1--36.

\bibitem[LLZ10]{leiloefflerzerbes10}
Antonio Lei, David Loeffler, and Sarah~Livia Zerbes, \emph{Wach modules and
  {I}wasawa theory for modular forms}, Asian J. Math. \textbf{14} (2010),
  no.~4, 475--528.

\bibitem[LV19a]{LongoVigni}
Matteo Longo and Stefano Vigni, \emph{Kolyvagin systems and {I}wasawa theory of
  generalized {H}eegner cycles}, Kyoto J. Math. \textbf{59} (2019), no.~3,
  717--746.

\bibitem[LV19b]{LongoVigni2}
\bysame, \emph{Plus/minus {H}eegner points and {I}wasawa theory of elliptic
  curves at supersingular primes}, Boll. Unione Mat. Ital. \textbf{12} (2019),
  no.~3, 315--347.

\bibitem[Mat21]{Matar2021}
Ahmed Matar, \emph{Plus/minus selmer groups and anticyclotomic
  $\mathbb{Z}_p$-extensions}, Research in Number Theory \textbf{7} (2021),
  no.~3.

\bibitem[Maz72]{mazur72}
Barry Mazur, \emph{Rational points of abelian varieties with values in towers
  of number fields}, Invent. Math. \textbf{18} (1972), 183--266.

\bibitem[MR04]{MazurRubin}
Barry Mazur and Karl Rubin, \emph{Kolyvagin systems}, Mem. Amer. Math. Soc.
  \textbf{168} (2004), no.~799, viii+96pp.

\bibitem[Pol03]{pollack03}
Robert Pollack, \emph{On the {$p$}-adic {$L$}-function of a modular form at a
  supersingular prime}, Duke Math. J. \textbf{118} (2003), no.~3, 523--558.

\bibitem[Pon20]{ponsinet}
Gautier Ponsinet, \emph{On the structure of signed {S}elmer groups}, Math. Z.
  \textbf{294} (2020), no.~3-4, 1635--1658.

\bibitem[PR87]{Perrin-Riou}
B.~Perrin-Riou, \emph{Fonctions {$L$} {$p$}-adiques, th\'eorie d'{I}wasawa et
  points de {H}eegner}, Bull. Soc. Math. France \textbf{115} (1987), no.~4,
  399--456.

\bibitem[PW11]{pollackweston11}
Robert Pollack and Tom Weston, \emph{On anticyclotomic $\mu$-invariants of
  modular forms}, Compos. Math. \textbf{1} (2011), 439--485.

\bibitem[Spr12]{sprung09}
Florian E.~Ito Sprung, \emph{Iwasawa theory for elliptic curves at
  supersingular primes: a pair of main conjectures}, J. Number Theory
  \textbf{132} (2012), no.~7, 1483--1506.

\end{thebibliography}
\end{document}